\documentclass[11pt]{amsart}

\usepackage{amsmath,amsfonts,amsthm,amssymb,graphicx,tikz,color}
\usepackage[all,2cell,ps]{xy}

\theoremstyle{plain}
\newtheorem{thm}{Theorem}[section]
\newtheorem{lemma}[thm]{Lemma}
\newtheorem{prop}[thm]{Proposition}
\newtheorem{cor}[thm]{Corollary}

\newtheorem{qtn}[thm]{Question}

\theoremstyle{definition}

\newtheorem{ex}[thm]{Example}

\newtheorem{rem}[thm]{Remark}

\theoremstyle{remark}

\DeclareMathOperator{\SL}{SL} \DeclareMathOperator{\PSL}{PSL} 
\DeclareMathOperator{\GL}{GL} \DeclareMathOperator{\PGL}{PGL}
 \DeclareMathOperator{\PO}{PO}
 \DeclareMathOperator{\SO}{SO}

\DeclareMathOperator{\SU}{SU}

 \DeclareMathOperator{\Ad}{Ad}

\DeclareMathOperator{\supp}{supp}

\DeclareMathOperator{\Isom}{Isom}

\DeclareMathOperator{\rank}{\mathrm{rank}}

\newcommand{\wh}{\widehat}

\newcommand{\bbC}{\mathbb{C}}

\newcommand{\bbH}{\mathbb{H}}

\newcommand{\bbP}{\mathbb{P}}

\newcommand{\bbR}{\mathbb{R}}

\newcommand{\bfA}{\mathbf{A}}
\newcommand{\bfB}{\mathbf{B}}

\newcommand{\bfG}{\mathbf{G}}
\newcommand{\bfH}{\mathbf{H}}
\newcommand{\bfI}{\mathbf{I}}
\newcommand{\bfJ}{\mathbf{J}}

\newcommand{\bfL}{\mathbf{L}}
\newcommand{\bfM}{\mathbf{M}}
\newcommand{\bfN}{\mathbf{N}}

\newcommand{\bfP}{\mathbf{P}}
\newcommand{\bfQ}{\mathbf{Q}}
\newcommand{\bfR}{\mathbf{R}}
\newcommand{\bfS}{\mathbf{S}}

\newcommand{\bfU}{\mathbf{U}}
\newcommand{\bfV}{\mathbf{V}}
\newcommand{\bfW}{\mathbf{W}}

\newcommand{\frako}{\mathfrak{o}}

\newcommand{\fraks}{\mathfrak{s}}

\newcommand{\gam}{\gamma}
\newcommand{\Gam}{\Gamma}

\newcommand{\Del}{\Delta}

\newcommand{\Lam}{\Lambda}

\DeclareMathOperator{\Tr}{Tr}

\newcommand{\wt}{\widetilde}

\newcommand{\N}{\ensuremath{\mathbb{N}}}

\newcommand{\calS}{\ensuremath{\mathcal{S}}}

\newcommand{\bH}{\ensuremath{\mathbf{H}}}

\newcommand{\ssm}{\smallsetminus}

\usetikzlibrary{arrows}

\title[Totally Geodesic Submanifolds]{Arithmeticity, Superrigidity, and Totally Geodesic Submanifolds}
\author[U Bader]{Uri Bader}
\address{Weizmann Institute of Science}
\email{bader@weizmann.ac.il}
\author[D Fisher]{David Fisher}
\address{Department of Mathematics\\Indiana University\\Bloomington, IN 47405}
\email{fisherdm@indiana.edu, nimimill@iu.edu}
\author[N Miller]{Nicholas Miller}
\author[M Stover]{Matthew Stover}
\address{Department of Mathematics\\Temple University\\Philadelphia, PA 19122}
\email{mstover@temple.edu}

\begin{document}

\begin{abstract}
Let $\Gam$ be a lattice in $\SO_0(n, 1)$. We prove that if the associated locally symmetric space contains infinitely many maximal totally geodesic subspaces of dimension at least $2$, then $\Gam$ is arithmetic. This answers a question of Reid for hyperbolic $n$-manifolds and, independently, McMullen for hyperbolic $3$-manifolds. We prove these results by proving a superrigidity theorem for certain representations of such lattices. The proof of our superrigidity theorem uses results on equidistribution from homogeneous dynamics and our main result also admits a formulation in that language.
\end{abstract}

\maketitle


\section{Introduction}\label{sec:Intro}

In this paper, a totally geodesic subspace of a finite volume hyperbolic manifold or orbifold will always mean a properly immersed, topologically closed, totally geodesic subspace. A totally geodesic subspace is \emph{maximal} if it is not properly contained in another proper totally geodesic subspace. The main result of this paper is:

\begin{thm}
\label{thm:main}
Let $\Gam$ be a lattice in $\SO_0(n, 1)$. If the associated locally symmetric space contains infinitely many maximal totally geodesic subspaces of dimension at least $2$, then $\Gam$ is arithmetic.
\end{thm}

This answers a question, first posed informally by Alan Reid in the mid-2000s. Independently, Curtis McMullen asked whether Theorem \ref{thm:main} is true in the setting of hyperbolic $3$-manifolds (see \cite[Qn.\ 7.6]{Curt} or \cite[Qn.\ 8.2]{MR}). Theorem \ref{thm:main} is also motivated in part by a question of Gromov and Piatetski-Shapiro \cite[Qn.\ 0.4]{GPS}. In a prior paper with J.-F.\ Lafont, the last three authors proved that a large class of nonarithmetic hyperbolic $n$-manifolds, including all the hybrids constructed by Gromov and Piatetski-Shapiro, have only finitely many maximal totally geodesic submanifolds \cite{FLMS}. This provided the first known examples of hyperbolic $n$-manifolds, $n \ge 3$, for which the collection of totally geodesic hypersurfaces is finite and nonempty. The case when $M$ is a closed hyperbolic $3$-manifold was very recently and independently proved by Margulis and Mohammadi \cite{MM}. Their proof and ours both use a superrigidity theorem to prove arithmeticity, but the superrigidity theorems and their proofs are quite different.

We now briefly give some applications of Theorem \ref{thm:main} and its proof. First, combining Theorem \ref{thm:main} with a theorem of Reid \cite{ReidKnot} we obtain the following.

\begin{cor}
\label{cor:knots}
Let $K$ be a knot in $S^3$ such that $S^3 \ssm K$ admits a complete hyperbolic structure. Then $S^3 \ssm K$ contains infinitely many immersed totally geodesic surfaces if and only if $K$ is the figure-eight knot.
\end{cor}

Combining Theorem \ref{thm:main} with results of Benoist--Oh \cite[Thm. 10.1]{BO}, Lee--Oh \cite[Thm. 1.9(3)]{LO}, and the classification of arithmetic hyperbolic $n$-manifolds (e.g., see \cite{Meyer}), we also obtain the following.

\begin{cor}
\label{cor:kleinian}{\ }
\begin{enumerate}
\item If $M$ is a geometrically finite hyperbolic $3$-manifold containing infinitely many totally geodesic surfaces with finite area, then $M$ has finite volume and $\pi_1(M)$ is arithmetic.

\item If $M$ is a convex cocompact hyperbolic $n$-manifold containing infinitely many maximal totally geodesic surfaces with finite area, then $M$ is compact and $\pi_1(M)$ is arithmetic.

\item If $n \ge 4$ is even and $M$ is a finite volume hyperbolic $n$-manifold, then $M$ is arithmetic if and only if it contains infinitely many totally geodesic hypersurfaces.

\end{enumerate}
\end{cor}

For convex cocompact acylindrical $3$-manifolds, this result already follows from work of McMullen--Mohammadi--Oh \cite{MMO, MMO2} and Theorem \ref{thm:main}. See \S\ref{ssec:FinalB} for discussion of (3) in odd dimensions.

Methods analogous to those used in the proof of Theorem \ref{thm:main} can also be used to show the following.

\begin{thm}
\label{thm:dehnfilling}
Let $M$ be a cusped hyperbolic $3$-manifold of finite volume with at least one torus cusp, and $N$ a hyperbolic $3$-manifold obtained by Dehn filling on some nonempty subset of the torus cusps of $M$. Then only finitely many totally geodesic surfaces in $N$ are isotopic to the image of a totally geodesic surface in $M$.
\end{thm}

If either $M$ or $N$ is nonarithmetic then this simply follows from Theorem \ref{thm:main}. However, there are examples where $M$ and $N$ are both arithmetic and some totally geodesic surface in $M$ remains totally geodesic in $N$, and hence Theorem \ref{thm:main} is not relevant. See \S\ref{ssec:FinalA} for the proof of Theorem \ref{thm:dehnfilling}, discussion, and examples.

\medskip

Our approach to proving Theorem \ref{thm:main} is inspired by the Margulis superrigidity and arithmeticity theorems \cite{MargulisICM, MargulisSuperrigidity}. The superrigidity theorem gives criteria for when a representation of $\Gamma$ extends to a representation of the ambient Lie group $G$. Arithmeticity is then deduced using these criteria to control the representations of $\Gam$ one obtains by varying embeddings of the adjoint trace field of $\Gam$ into other local fields. See \S\ref{ssec:Deduce} for more discussion. A famous example of this strategy is the proof by Margulis of arithmeticity of lattices with \emph{dense commensurator} \cite{MargulisICM}. This theorem also holds in rank one and is the full converse to a theorem of Borel \cite{BorelCrelle}. Margulis proved this by classifying representations of lattices that extend to representations of some dense subgroup of $G$ contained in the commensurator.

Relating dense commensurators of arithmetic lattices back to the existence of infinitely many totally geodesic submanifolds, one can easily observe:

\medskip
\noindent
\textbf{Arithmetic geodesic submanifold dichotomy}: \emph{For any $1 \le k \le n-1$, an arithmetic hyperbolic $n$-manifold either contains no codimension $k$ geodesic submanifolds, or it contains infinitely many and they are everywhere dense.}

\medskip

This observation is one of the motivations for the question answered by Theorem \ref{thm:main} and was perhaps first made precise in dimension $3$ by Maclachlan--Reid and Reid \cite{MRTG, ReidTG}, who also exhibited the first hyperbolic $3$-manifolds with no totally geodesic surfaces. Note that an analogous statement holds for any arithmetic locally symmetric space. See \cite{FLMS} for further discussion and examples.

Our proof of Theorem \ref{thm:main} rests on two key points:
\begin{enumerate}
\item From certain homomorphisms $\rho: \Gamma \rightarrow H$, we construct a good measure on a fiber bundle over $G/\Gamma$ that is invariant under a proper noncompact connected simple subgroup $W<G$. This is accomplished in \S\ref{section:measuresandarithmeticity}.
\item A superrigidity theorem showing that the measure constructed in $(1)$ allows us to extend $\rho$, provided that $H$ satisfies an additional {\em compatibility} condition. This is proved in \S\ref{section:proofs}.
\end{enumerate}

\noindent In the standard language of superrigidity and its proofs, one can view $(1)$ as the analogue for constructing a boundary map and $(2)$ as the analogue for using the boundary map to show that the representation $\rho$ extends.

We now discuss each of these steps briefly and begin by stating a version of Theorem \ref{thm:main} in language from homogeneous dynamics. We consider a proper noncompact connected closed simple subgroup $W<G=\SO_0(n,1)$. Then $W$ is isomorphic to $\SO_0(m,1)$ for some $1< m < n$. We have a $W$-action on $G/\Gamma$, and results of Ratner classify the $W$-invariant ergodic measures for this action \cite{RatnerDuke}. We say a measure $\nu$ on $G/\Gamma$ has {\em proper support} if its support is a proper closed subset.

\begin{thm}
\label{thm:main:homogeneous}
If there exists an infinite sequence $\{\mu_i\}$ of $W$-invariant, ergodic measures with proper support for which Haar measure on $G/\Gamma$ is a weak-$*$ limit of the $\mu_i$, then $\Gamma$ is arithmetic.
\end{thm}

We show in Proposition \ref{prop:homogen} that Theorem \ref{thm:main:homogeneous} implies Theorem \ref{thm:main}. 

\medskip

In proving arithmeticity we are given a local field $k$ of characteristic zero, a connected semisimple adjoint $k$-algebraic group $\bfH$ with $k$-points $\bfH(k)$, and a representation $\rho: \Gam \rightarrow \bfH(k)$. We consider a certain irreducible representation of $\bfH(k)$ on a finite dimensional $k$-vector space $V$ and the associated projective space $\bbP(V)$. We then use the hypotheses of either Theorem \ref{thm:main} or Theorem \ref{thm:main:homogeneous} to build a $W$-invariant ergodic measure on the bundle $(G\times\bbP(V))/\Gam$ that projects to Haar measure on $G/\Gamma$.

We now state the superrigidity theorem that finishes the proof from the existence of such a measure. This requires an additional technical assumption on the pair $k$ and $\bfH$. Let $P$ be a minimal parabolic subgroup of $G$ and $U$ its unipotent radical. A pair consisting of a local field $k$ and a $k$-algebraic group $\bfH$ is said to be \emph{compatible} with $G$ if for every nontrivial $k$-subgroup $\bfJ<\bH$ and any continuous homomorphism $\tau:P \rightarrow N_{\bfH}(\bfJ)/\bfJ(k)$, where $N_{\bfH}(\bfJ)$ is the normalizer of $\bfJ$ in $\bfH$, we have that the Zariski closure of $\tau(U')$ coincides with the Zariski closure of $\tau(U)$ for every nontrivial subgroup $U'<U$ (see \S\ref{ssec:Compatible}).

\begin{thm}
\label{theorem:superrigiditydichotomy}
Let $G$ be $\SO_0(n, 1)$ for $n \ge 3$, $W<G$ be a noncompact simple subgroup, and $\Gamma < G$ be a lattice. Suppose that $k$ is a local field and $\bH$ is a connected $k$-algebraic group such that the pair consisting of $k$ and $\bfH$ is compatible with $G$. Finally, let $\rho: \Gamma \rightarrow \bH(k)$ be a homomorphism with unbounded, Zariski dense image. If there exist a $k$-rational faithful irreducible representation $\bfH \to \SL(V)$ on a $k$-vector space $V$ and a $W$-invariant measure $\nu$ on $(G\times \mathbb{P}(V))/\Gamma$ that projects to Haar measure on $G / \Gam$, then $\rho$ extends to a continuous homomorphism from $G$ to $\bH(k)$.
\end{thm}

\begin{rem} \label{rem:SU}
We state the theorem for $G=\SO_0(n,1)$ for simplicity, but the same theorem holds, with practically the same proof, for every connected simple $\bbR$-rank one Lie group. In particular, there is an analogue of Theorem \ref{theorem:superrigiditydichotomy} for lattices in $\SU(n,1)$. 
\end{rem}

Understanding invariant measures for dynamical systems that are not homogeneous plays an important role in other recent results in rigidity theory. For example, see work of Brown, Hurtado, and the second author on Zimmer's conjecture \cite{BrownFisherHurtado,BrownFisherHurtado2}. In that context, Theorem \ref{theorem:superrigiditydichotomy} can be thought of as classifying invariant measures in a nonhomogeneous setting. Indeed, Theorem \ref{theorem:superrigiditydichotomy} shows that either there is no extension of $\rho$ and hence no such $W$-invariant measures exist, or there is a simple classification of all invariant measures on the projective bundle.

We note in closing that Theorem \ref{theorem:superrigiditydichotomy} can be reformulated in several equivalent ways. There is also an analogous \emph{superrigidity for cocycles} that follows from the same proof, and which provides some partial technical results towards questions raised by results of Zimmer and Bader--Furman--Sauer \cite{ZimmerCSR, BFS}.

\subsubsection*{Acknowledgments}
Bader was supported in part by the ISF grant 704/08.
Fisher was partially supported by NSF Grant DMS-1607041, the Institute for Advanced Study, and the Simons Collaboration on Algorithms and Geometry. Miller was partially supported from U.S. National Science Foundation grants DMS 1107452, 1107263, 1107367 ``RNMS: GEometric structures And Representation varieties'' (the GEAR Network). Stover was partially supported by Grant Number 523197 from the Simons Foundation/SFARI.

Our approach to Theorem \ref{thm:main} owes a tremendous debt to the ideas of Gregory Margulis on superrigidity and arithmeticity \cite{MargulisICM, MargulisSuperrigidity, MargulisBook}. The authors thank Ian Agol, Matt Bainbridge, Alex Eskin, Michael Larsen, Homin Lee, Alan Reid, and Dave Witte Morris for helpful conversations. They particularly thank Hee Oh for detailed comments on an earlier draft, Jean-Fran\c cois Lafont for his participation in the early phases of this project, and Alex Furman for his inspiring work on superrigidity with the first author.

\section{Fixed Notation} \label{sec:FN}

We first fix some notation that will be used throughout our paper. Let $\bfG_0$ denote $\SO(n,1)$ for $n\geq 3$, considered as a real algebraic group. We let $G$ be the connected component of the identity in $\bfG_0(\bbR)$, that is, $G= \SO_0(n,1)$. Set $K=\SO_0(n)<G$ and identify $K\backslash G$ with hyperbolic $n$-space. For a noncompact simple subgroup $W<G$, fix a maximal $\bbR$-split torus $A < W$. Since $W$ and $G$ are both $\bbR$-rank one, $A$ is also a maximal $\bbR$-split torus of $G$. Fix a maximal unipotent subgroup $U$ of $G$ normalized by $A$ and let $M$ be the compact factor of the Levi decomposition of the connected component of the identity in the centralizer of $A$. Then $P = M A U$ is the \emph{Langlands decomposition} of the maximal parabolic subgroup of $G$ associated with the pair $(A, U)$. Set $U^\prime = W \cap U$, and note that it is a maximal unipotent subgroup of $W$.

Now, fix a lattice $\Gam < G$. When considering the action of $\Gam$ on $G$, we always consider the right action, $g \cdot \gam = g \gam^{-1}$ and $X_\Gamma=K\backslash G/\Gamma$ denotes the corresponding locally symmetric space. Let $\ell$ be the trace field of $\Gamma$, that is the subfield of $\bbR$ generated by all elements of the form $\Tr(\Ad(\gamma))$ for $\gamma \in \Gamma$, where $\Ad$ denotes the adjoint representation. Denote the inclusion of $\ell$ in $\bbR$ by $w:\ell\to \bbR$. By work of Vinberg \cite{VinbergDef}, there exists an $\ell$-group $\bfG$ and an $\bbR$-isogeny $\bfG\to \bfG_0$ such that the image of $\bfG(\ell)$ in $\bfG_0(\bbR)$ contains a finite index subgroup of $\Gamma$. Passing to this finite index subgroup, we will assume throughout that $\Gamma$ is contained in the image of $\bfG(\ell)$. By \cite{VinbergDef}, $\ell$ is the minimal field of definition of $\Gam$. Moreover, it follows from work of Selberg, Calabi, Raghunathan, and Garland \cite{Selberg, Calabi, RaghunathanRigid, Garland} that $\ell$ is in fact a number field.

\section{Finding invariant measures and arithmeticity}
\label{section:measuresandarithmeticity}

In this section we show how Theorem \ref{theorem:superrigiditydichotomy} implies Theorem \ref{thm:main}. 
We show in \S\ref{ssec:Orbits} that the hypotheses of Theorem~\ref{thm:main:homogeneous} are implied by the hypotheses of Theorem~\ref{thm:main},
\S\ref{ssec:Deduce} recalls the overall strategy of deducing arithmeticity from superrigidity, \S\ref{Ssec:buildingmeasures} finds the measure $\nu$ assumed in Theorem \ref{theorem:superrigiditydichotomy} using the hypotheses of Theorem \ref{thm:main:homogeneous}, and finally \S\ref{ssec:Compatible} shows that all the target groups considered for proving arithmeticity are compatible. In particular, this section reduces Theorem \ref{thm:main} to Theorem \ref{thm:main:homogeneous} and Theorem \ref{thm:main:homogeneous} to Theorem \ref{theorem:superrigiditydichotomy}.

\subsection{Geodesic submanifolds and properly supported measures}\label{ssec:Orbits}

Recall that a finite measure $\mu$ on $G/\Gamma$ is called \emph{homogeneous} if there is a closed subgroup $S<G$ such that $\mu$ is Haar measure on a closed $S$-orbit in $G/\Gamma$. Such a homogeneous measure is said to be $W$-ergodic when $W$ is a closed subgroup of $S$ under which $\mu$ is ergodic. 
In this case, the support of the measure is said to be a \emph{$W$-ergodic homogeneous subspace} of $G/\Gamma$.
For $1< m \le n$ we let $W_m\le G$ be the standard embedding of $\SO_0(m,1)$ into $G$. The entirety of this subsection is devoted to proving the following proposition.

\begin{prop}
\label{prop:homogen}
For the real hyperbolic space $X_\Gamma = K \backslash G /\Gamma$ the following are equivalent:
\begin{enumerate}
  \item $X_\Gamma$ contains infinitely many maximal totally geodesic subspaces of dimension two or higher,
  \item for some $1<m<n$, there exists an infinite sequence $\{\mu_i\}$ of $W_m$-invariant, ergodic measures with proper support for which Haar measure on $G/\Gamma$ is a weak-$*$ limit of the $\mu_i$,
  \item for some $1<m<n$, there exists an infinite sequence of homogeneous, $W_m$-ergodic measures $\{\mu_i\}$ for which Haar measure on $G/\Gamma$ is a weak-$*$ limit of the $\mu_i$.
\end{enumerate}
\end{prop}

That $(3)$ implies $(2)$ is clear and the reverse implication is a theorem of Ratner \cite{RatnerOrbitClosure} (see also Einsiedler \cite{Einsiedler}). It therefore suffices to show that $(1)$ and $(3)$ are equivalent. Throughout this section we let $\pi:G/\Gamma\to X_\Gamma$ be the natural projection.

We start by clarifying the relationship between totally geodesic subspaces and homogeneous measures. We first recall that a subspace $Z$ of a hyperbolic $n$-orbifold $X$ is totally geodesic if it is properly immersed and if one (hence any) lift to a map of orbifold universal covers $\widetilde{Z} \to \widetilde{X}$ is a totally geodesic isometric embedding of hyperbolic $m$-space in hyperbolic $n$-space for some $m \le n$. In particular, totally geodesic subspaces are by definition connected.

\begin{lemma} \label{lem:clarifying}
Fix $1< m \le n$ and $X_\Gamma=K\backslash G/\Gamma$. Then the following hold.
\begin{enumerate}

\item 
Let $S\le G$ be a closed subgroup containing $W_m$ and $h \in G$ be such that $Sh\Gamma/\Gamma \subset G/\Gamma$ is a closed $S$-orbit. 
Then the subspace $Z=\pi(Sh\Gamma/\Gamma)$ of $X_\Gamma$ is a closed totally geodesic $m^\prime$-dimensional subspace for some $m^\prime \geq m$ and, up to normalization, the $m^\prime$-volume of $Z$ is the push-forward of the corresponding homogeneous measure on $G / \Gamma$. \label{lem:clarifying1}

\item Furthermore, under the assumption above, $m^\prime=m$ if and only if all unipotent elements of $S$ are contained in $W_m$. In this case $S$ is a subgroup of the normalizer $N_m$ of $W_m$ in $G$, and $N_mh\Gamma/\Gamma\subset G/\Gamma$ is also closed with projection $\pi(N_mh\Gamma/\Gamma)=Z$. \label{lem:clarifying2}

\item Conversely, every $m$-dimensional closed totally geodesic subspace $Z$ in $X_\Gamma$ has finite $m$-volume and moreover $Z=\pi(Sh\Gamma/\Gamma)$ for some closed intermediate subgroup $W_m \le S \le N_m$ and some homogeneous, $W_m$-ergodic subspace $Sh\Gamma/\Gamma \subset G/\Gamma$. \label{lem:clarifying3}

\end{enumerate}
\end{lemma}

\begin{proof}
We start by observing that for each $m^\prime$, the image of $W_{m^\prime}$ in $K\backslash G$, namely $K\backslash KW_{m^\prime}$, is an $m^\prime$-dimensional closed totally geodesic subspace. As $G$ acts transitively on the collection of $m^\prime$-dimensional closed totally geodesic subspaces of $K\backslash G$, any such subspace is of the form $K\backslash KW_{m^\prime}g$ for some $g\in G$. Since $N_{m^\prime}$ is contained in $KW_{m^\prime}$, every intermediate subgroup $W_{m^\prime} \le S_1 \le N_{m^\prime}$ has the property that $K\backslash KS_1g=K\backslash KW_{m^\prime}g$ is an $m^\prime$-dimensional closed totally geodesic subspace of $K\backslash G$, and the push-forward of the volume form on the Lie group $S_1$ is its $m^\prime$-volume. Conversely, $N_{m^\prime}$ is exactly the stabilizer of $K\backslash KW_{m^\prime}$ in $G$, thus if $K\backslash KS_1g$ is an $m^\prime$-dimensional closed totally geodesic subspace of $K\backslash G$ then $W_{m^\prime} \le S_1 \le N_{m^\prime}$.

We now prove part \ref{lem:clarifying1}. Let $S \le G$ be a closed subgroup that contains $W_m$ and let $h\Gamma/\Gamma \in G/\Gamma$ be a point whose $S$-orbit is closed. Denote by $S^+ \le S$ the closed normal subgroup generated by unipotent elements in $S$. Then $S^+$ is a connected semisimple subgroup of $G$ which contains $W_m$, hence it is conjugate to $W_{m^\prime}$ for some $m^\prime\ge m$. In fact, if $C_m$ denotes the centralizer of $W_m$ in $G$, it is straightforward to see that there exists $g\in C_m$ such that $W_{m^\prime}=(S^+)^g=gS^+g^{-1}$. We fix such a $g$ and set $S_1=S^{g}$.

Since $S^+$ is normal in $S$, $W_{m^\prime}$ is normal in $S_1$, and thus $W_{m^\prime} \le S_1 \le N_{m^\prime}$. From the fact that $ C_m < K$ and $g\in C_m$, we get that
\[
KS=Kg^{-1}S_1g=KS_1g.
\]
Since the projection $\pi$ is proper, $Z=\pi(Sh\Gamma/\Gamma)$ is closed,
and since
\[
Z=K\backslash KSh\Gamma/\Gamma = K\backslash KS_1gh\Gamma/\Gamma \subset X_\Gamma,
\]
we conclude that $Z$ is $m^\prime$-dimensional and totally geodesic, since it is the image of $K\backslash KS_1gh$ under the projection $K\backslash G\to X_\Gamma$. Consequently, the $m^\prime$-volume on $Z$ is the push forward of the $S_1$-volume form. As the $S_1$-volume form is the $g^{-1}$-conjugate of the $S$-volume form, the $m^\prime$-volume of $Z$ is the push forward of Haar measure on $Sh\Gamma/\Gamma$. This completes the proof of part \ref{lem:clarifying1}.

We now prove part \ref{lem:clarifying2}. If all unipotent elements of $S$ are contained in $W_m$ then $S^+ \le W_m$, but $W_m \le S^+$ by hypothesis, so we conclude that $W_m=S^+$ and $m^\prime=m$. Conversely, if $m^\prime=m$ then since $W_m \le S^+ = W_{m^\prime}^g$ we conclude that $S^+=W_m$ and all unipotent elements of $S$ are contained in $W_m$. In this case, $W_m$ is normal in $S$ and thus $W_m \le S \le N_m$. Since $W_m$ is cocompact in $N_m$, $S$ is as well, and therefore the fact that $Sh\Gamma/\Gamma$ is closed in $G/\Gamma$ implies the same holds for $N_mh\Gamma/\Gamma$.
Moreover, from the chain of equations
\[
K\backslash KW_mh=K\backslash KSh=K\backslash KN_mh,
\]
we conclude that $\pi(N_m h\Gamma/\Gamma)=Z$ in $X_\Gamma$. This proves part \ref{lem:clarifying2}.

Before turning to part \ref{lem:clarifying3}, we discuss $m$-dimensional closed totally geodesic immersed submanifolds of $n$-dimensional Riemannian manifolds in general. Given such a pair $M\subset N$ we let $F(N)$ be the oriented orthonormal frame bundle of $N$ and we let $F_N(M)$ be the subbundle where the fiber over each point $x\in M$ is the subset of frames in $F(N)_x$ whose first $m$ vectors are tangent to $Z$. We note that $F_N(M)$ is a principal $\mathrm{S}(\mathrm{O}(m)\times \mathrm{O}(n-m))$-bundle over $M$ and it is closed in $F(N)$, which itself is a principal $\SO(n)$-bundle over $N$. This construction is natural under covering maps.

Identifying $G$ with $F(K\backslash G)$, one checks easily that $N_m$ gets identified with $F_{K\backslash G}(K\backslash KW_m)$ and thus for every $g\in G$, $N_mg$ gets identified with $F_{K\backslash G}(K\backslash KW_mg)$. In accordance with the identification of $G$ with $F(K\backslash G)$ we identify $G/\Gamma$ with $F(X_\Gamma)$. For an $m$-dimensional closed totally geodesic subspace $Z \subset X_\Gamma$ the subbundle $F_{X_\Gamma}(Z)$ gets identified with a closed subset of $G/\Gamma$. Finding $g\in G$ such that $K\backslash KW_mg\subset K\backslash G$ projects to $Z$ under the natural map $K\backslash G\to X_\Gamma$, we conclude by the naturality under covering maps that $F_{X_\Gamma}(Z)$ is identified with $N_mg\Gamma/\Gamma$. In particular, the latter is a closed subset of $G/\Gamma$ whose image under $\pi$ is $Z$.

We are now in a position to prove part \ref{lem:clarifying3}. Let $Z \subset X_\Gamma$ be an $m$-dimensional closed totally geodesic subspace. The fact that it has a finite measure is well-known; see \cite[Prop.~3.4]{Gelander-Levit} for a recent reference. By the discussion above there exists $g\in G$ such that $N_mg\Gamma/\Gamma$ is a closed subset of $G/\Gamma$ whose image under $\pi$ is $Z$. Note that $N_mg\Gamma/\Gamma$ has a finite volume, as it is a compact extension of $Z$. Since $W_m \le N_m$, $N_m g\Gamma/\Gamma$ is $W_m$-invariant, though it might not be $W_m$-ergodic. Fix a $W_m$-ergodic measure $\mu$ in its ergodic decomposition and let $S$ be the stabilizer of $\mu$ in $N_m$. Clearly, $W_m \le S \le N_m$ and $\mu$ is $S$-homogeneous by Ratner's Theorem. 
Let $h\in G$ be such that $\mu$ is the Haar measure of the $S$-orbit $Sh\Gamma/\Gamma$.
It follows from part 1 of the lemma that $\pi(Sh\Gamma/\Gamma)$ is a closed totally geodesic subspace of $X_\Gamma$ of dimension $m$ or higher. Since this image is contained in the $m$-dimensional closed totally geodesic subspace $Z$ it must coincide with it. This completes the proof of the lemma.
\end{proof}

We will use the following theorem which follows from combining the existence of compact cores for hyperbolic manifolds with works of Dani--Margulis
\cite[Thm.~6.1]{DaniMargulis} and Mozes--Shah \cite{MS}.

\begin{thm} \label{thm:MS}
Suppose that $W< G$ is a closed connected semisimple subgroup that is generated by unipotent elements. Let $\{\mu_i\}$ be a sequence of homogeneous, $W$-ergodic probability measures on $G/\Gamma$ that weak-$*$ converges in the space of all finite Radon measures to a measure $\mu$. Then $\mu$ is a homogeneous, $W$-ergodic probability measure on $G/\Gamma$ and there exists a sequence $\{g_i\}$ in $G$ and a natural number $i_0$ such that, for each $i\geq i_0$, the measure $g_i\mu$ is a  homogeneous, $W$-ergodic probability measure on $G/\Gamma$ whose support contains the support of $\mu_i$.
\end{thm}

\begin{proof}
We first claim that $\mu$ is not the zero measure. This is trivial if $G/\Gamma$ is compact. As this claim is conjugation invariant, we will assume as we may that $W=W_m$ for some $m>1$. Using \cite[Lem.~5.13]{FLMS} we fix a compact set $C_1$ in $X_\Gamma=K\backslash G/\Gamma$ that meets every closed totally geodesic subspace of dimension at least $2$.
Then, choose a compact set $C_2$ that contains $C_1$ in its interior and consider its preimage in $G / \Gam$, i.e., the compact subset $F=\pi^{-1}(C_2)$ of $G/\Gamma$. Fix a 1-parameter unipotent subgroup $\{u_t\}$ in $ W$ and set $\epsilon=1/2$. Applying \cite[Thm.~6.1]{DaniMargulis} we find a compact subset $F^\prime\subset G/\Gamma$ such that
\begin{equation} \label{eq:DM}
\frac{1}{T}\int_0^T \chi_{F^\prime}(u_ty) dt=
\frac{1}{T}\lambda\left\{ t\in [0,T] \mid u_t y\in F^\prime\right\} \geq \frac{1}{2},
\end{equation}
for every $y\in F$ and every $T\ge 0$, where $\chi_{F^\prime}$ is the characteristic function of $F^\prime$ and $\lambda$ is the Lebesgue measure on $\mathbb{R}$. The claim will follow once we show that $\mu_i(F^\prime) \geq 1/2$ for every $i$, thus $\mu(F^\prime)\geq 1/2$. We now fix $i$ and show that indeed $\mu_i(F^\prime) \geq 1/2$.

By Lemma~\ref{lem:clarifying}\eqref{lem:clarifying1}, $\pi_*\mu_i$ is the unit renormalization of the volume measure associated with a closed totally geodesic subspace of dimension at least $m\ge 2$. This subspace intersects $C_1$ nontrivially, by the choice of $C_1$, thus it intersects $C_2$ in an open set. It follows that $\mu_i(F)=\pi_*\mu_i(C_2)>0$, as $\pi_*\mu_i$ is proportional to a volume measure. We note that $\mu_i$ is $\{u_t\}$-ergodic by the Howe--Moore theorem, and we let $y\in F$ be a $\{u_t\}$-generic point with respect to $\mu_i$. Applying the Birkhoff ergodic theorem to the function $\chi_{F'}$ we conclude by Equation~\eqref{eq:DM} that indeed $\mu_i(F') \geq 1/2$. This proves the claim.

Using \cite[Cor.~1.3]{MS} we therefore conclude that $\mu$ is a homogeneous, $S$-ergodic probability measure on $G/\Gamma$, where $S$ is the subgroup of the stabilizer of $\mu$ generated by unipotent elements. The group $S$ is not unipotent, as it contains $W$, therefore it must be semisimple. We conclude by the Howe--Moore theorem that $\mu$ is $W$-ergodic.

Let $Y_i=\supp(\mu_i)$ and $Y=\supp(\mu)$ in $G/\Gamma$ and fix one dimensional unipotent subgroups $U_1,\ldots,U_k<W$ that generate $W$. Note that $\mu_i$ is $U_j$-ergodic for every $i$ and every $1\leq j\leq k$ by the Howe--Moore theorem. Thus, for any fixed $i$, the subset $Y^\prime_i$ consisting of the points in $Y_i$ that are $U_j$-generic for every $1\le j\le k$ is of full $\mu_i$-measure, hence it is dense in $Y_i$. We fix a point $y\in Y$. As $\mu=\lim \mu_i$, we can find a sequence of points $\{y_i\}$ converging to $y$ such that $y_i$ is in $Y_i$ for every $i$. By deforming such a sequence, using the fact that $Y^\prime_i$ is dense in $Y_i$, we find and fix a sequence $\{y^\prime_i\}$ converging to $y$ such that $y^\prime_i$ is in $Y^\prime_i$ for every $i$. 
We fix a sequence $\{g_i\}$ in $G$ such that $\lim g_i=e$ and $g_iy=y^\prime_i$ for all $i$.

By \cite[Thm.~1.1]{MS}, we conclude that for every $1\le j \le k$ there exists an $i_j$ such that $Y_i\subset g_iY$ and $\mu$ is $U_j^{g_i^{-1}}$-invariant for every $i\geq i_j$. Let $i_0=\max \{i_j\}$ and fix $i\geq i_0$. As $g_iY=\supp(g_i\mu)$, we conclude that $\supp(\mu_i)\subseteq\supp(g_i\mu)$. It remains to show that $g_i\mu$ is $W$-ergodic or, equivalently, that $\mu$ is $W^{g_i^{-1}}$-ergodic. Since $\mu$ is $U_j^{g_i^{-1}}$-invariant for every $1\le j\le k$, it is $W^{g_i^{-1}}$-invariant. It follows that $W^{g_i^{-1}}$ is contained in $S$. Then $W^{g_i^{-1}}$-ergodicity follows from the Howe--Moore theorem and $S$-ergodicity. This completes the proof.
\end{proof}

\begin{proof}[Proof of Proposition \ref{prop:homogen}]
As noted immediately after the statement of the proposition, it is enough to show that (1) and (3) are equivalent. We begin with the easier implication, namely that (3) implies (1). Fix $1<m<n$, and suppose that $\{\mu_i\}$ is a sequence of homogeneous, $W_m$-ergodic measures for which Haar measure on $G/\Gamma$ is a weak-$*$ limit. Let $\bar{\mu}_i=\pi_*\mu_i$ be the push-forward measures. By Lemma~\ref{lem:clarifying}\eqref{lem:clarifying1}, each measure $\bar{\mu}_i$ is supported on a closed totally geodesic subspace of $X_\Gamma$ and we let $Z_i$ be a maximal totally geodesic subspace of $X_\Gamma$ containing it. Since Haar measure of $G/\Gamma$ is by hypothesis a weak-$*$ limit of the sequence $\{\mu_i\}$, its push-forward is supported on the closure of $\bigcup Z_i$. Since the push forward of Haar measure on $G/\Gamma$ is the volume form on $X_\Gamma$, it follows that $\bigcup Z_i$ is dense in $X_\Gamma$ and hence the sequence $\{Z_i\}$ consists of infinitely many maximal totally geodesic subspaces. This implies (1).

Next we show that (1) implies (3). Assume that there exists an infinite sequence $\{Z_i\}$ of distinct closed maximal totally geodesic subspaces of $X_\Gamma$. By passing to a subsequence we assume that they all have the same dimension $m$ for some $1<m<n$. By Lemma~\ref{lem:clarifying}\eqref{lem:clarifying3}, each $Z_i$ is the image under $\pi : G / \Gam \to X_\Gam$ of a homogeneous, $W_m$-invariant subspace of $G/\Gamma$ which we denote by $Y_i$. Furthermore each $Y_i$ is the support of a homogeneous, $W_m$-ergodic probability measure $\mu_i$ with stabilizer $S_i$ containing $W_m$. Passing to a further subsequence, we can assume that the $\mu_i$ weak-$*$ converge to a measure $\mu$. Note that Theorem~\ref{thm:MS} implies that $\mu$ is a homogeneous, $W_m$-ergodic probability measure on $G/\Gamma$. We now want to show that $\mu$ is Haar measure on $G/\Gamma$.

Assume for contradiction that $\mu$ is not Haar measure on $G/\Gamma$. 
If $S$ denotes the stabilizer of $\mu$, which contains $W_m$, then $Y_\infty=\supp(\mu)=Sh\Gamma/\Gamma$ for some $h\in G$. 
By Theorem~\ref{thm:MS} there exists a sequence $\{g_i\}$ in $G$ and a natural number $i_0$ such that for each $i\geq i_0$, $g_iY_\infty$ is a homogeneous, $W_m$-invariant subspace of $G/\Gamma$ that contains $Y_i$.
Once again passing to a subsequence we assume that this holds for every $i \ge 1$.

Fix any $i\in\N$. Applying Lemma~\ref{lem:clarifying}\eqref{lem:clarifying1} to $g_iY_\infty$, we see that $\pi(g_iY_\infty)$ is a proper closed totally geodesic subspace, where properness follows from our assumption that $\mu$ is not Haar measure along with the special case $m^\prime=n$ in Lemma~\ref{lem:clarifying}\eqref{lem:clarifying1}. Since $Y_i\subseteq g_iY_\infty$, we deduce that $Z_i\subseteq \pi(g_iY_\infty)$ and hence $Z_i= \pi(g_iY_\infty)$ by maximality.

Note that $g_iY_\infty$ is an $S^{g_i}$-homogeneous space and hence $W_m \le S^{g_i}$. By Lemma~\ref{lem:clarifying}\eqref{lem:clarifying2} we conclude that the subgroup of $S^{g_i}$ generated by unipotent elements is $W_m$ and that $W_m\le S^{g_i} \le N_m$. Since $W_m \le S$, we also have that $W_m^{g_i} \le S^{g_i}$ and since the subgroup of $N_m$ generated by unipotents is exactly $W_m$, we see that $W_m=W_m^{g_i}$. Therefore $g_i\in N_m$.
Applying Lemma~\ref{lem:clarifying}\eqref{lem:clarifying1} to the closed $S^{g_i}$-orbit
\[
S^{g_i}g_ih\Gamma/\Gamma=g_iSh\Gamma/\Gamma=g_i Y_\infty,
\]
we see that the $N_m$-orbit $N_mg_ih\Gamma/\Gamma$ is also closed in $G/\Gamma$ and that we have $Z_i=\pi(N_mg_ih\Gamma/\Gamma)$. However $g_i\in N_m$, thus $N_mg_ih\Gamma/\Gamma=N_mh\Gamma/\Gamma$ is independent of $i$. We conclude that $Z_i=\pi(N_mh\Gamma/\Gamma)$ is independent of $i$, contradicting the assumption that the spaces $Z_i$ are all distinct. This contradiction concludes the proof that (1) implies (3).
\end{proof}

\subsection{The proofs of Theorem \ref{thm:main} and Theorem \ref{thm:main:homogeneous}}
\label{ssec:Deduce}

We now explain how to prove Theorem \ref{thm:main} and Theorem \ref{thm:main:homogeneous} given Theorem \ref{theorem:superrigiditydichotomy}. This closely follows Margulis's proof of arithmeticity from superrigidity. For more details see \cite[Ch.\ IX]{MargulisBook} or \cite[Ch.\ 6]{Zimmer}.

We are given a lattice $\Gam < G$ and want to show that it is arithmetic. As in \S\ref{sec:FN}, we consider $\Gamma$ as a subgroup of $\bfG(\ell)$, where $\ell$ is the adjoint trace field of $\Gam$, embedded in $\bbR$ via $w:\ell\to\bbR$. Consider the collection $S$ of all places of $\ell$, that is, the equivalence classes of dense embeddings of $\ell$ into a local field. For $v\in S$, $\ell_v$ will denote the corresponding local field. In particular we have the aforementioned $w\in S$ and $w:\ell\to \ell_w=\bbR$. Considering the various embeddings $\Gamma \to \bfG(\ell)\to \bfG(\ell_v)$ for all $v\in S$, it is standard that $\Gamma$ is arithmetic if and only if the image of $\Gam$ in $\bfG(\ell_v)$ is precompact for every $v\neq w$.

We let $\bfH$ be the adjoint group associated with $\bfG$ and claim that for $v\in S$, $v\neq w$, the corresponding homomorphism $\Gamma\to \bfG(\ell_v)\to \bfH(\ell_v)$ cannot be extended to $\bfG(\ell_w) \simeq G\to \bfH(\ell_v)$. By \cite[Rem.\ 1.8.2(III)]{MargulisBook}, such an extension gives rise to a continuous field embedding $\ell_w\to \ell_v$ and this field embedding clearly agrees with $v:\ell\to \ell_v$ on the set of elements of the form $\Tr(\Ad(\gamma))$ for $\gamma\in \Gamma$. As $\ell$ is generated by the above set, we get that $\ell_w\to \ell_v$ extends $v$, which contradicts the assumption that $v\neq w$. To be precise, in \cite[Rem.\ 1.8.2(III)]{MargulisBook} the target group is assumed absolutely simple, which is not necessarily the case for $\bfH$. This can be remedied by passing to a certain finite field extension $\ell'_v/\ell_v$, considering the corresponding homomorphism $\bfG(\ell_v)\to \bfH(\ell_v) \to \bfH(\ell'_v)$, taking the restriction of scalars of $\bfH$ from $\ell'_v$ back to $\ell_v$, then projecting to a simple factor. This procedure replaces the target group $\bfH$ with an absolutely simple $\ell_v$-group, thus proving our claim by the argument presented above. In summary, we prove that $\Gamma$ is arithmetic by showing that its image in $\bfG(\ell_v)$ is precompact, and we do that by proving that if this is not the case then $\Gamma\to \bfH(\ell_v)$ must extend to $G$. Note that the failure of precompactness of the image of $\Gamma$ in $\bfG(\ell_v)$ implies the same holds for the image of $\Gamma$ in $\bfH(\ell_v)$, as the map $\bfG\to \bfH$ is a finite isogeny.

The existence of the desired extension $G\to \bfH(\ell_v)$ will follow from Theorem~\ref{theorem:superrigiditydichotomy} once we verify its various assumptions in the specific settings of Proposition~\ref{prop:homogen}. In this setting, in \S\ref{Ssec:buildingmeasures} we will produce an $\ell_v$-vector space $V$, endowed with a faithful irreducible representation of $\bfH(\ell_v)$, and a $W$-invariant measure on $(G\times\bbP(V))/\Gamma$, as required in Theorem~\ref{theorem:superrigiditydichotomy}. Our proof will be complete once we show that the pair consisting of $\ell_v$ and $\bfH$ is compatible with $G$. This will be done in \S\ref{ssec:Compatible}.

\subsection{Lifting measures to the projective bundle}
\label{Ssec:buildingmeasures}

Let $\ell$ be the number field, $\bfG$ the $\ell$-algebraic group associated with $\Gam$ as in \S\ref{sec:FN}, and $\bfH$ be the corresponding adjoint $\ell$-group. In this section we let $k = \ell_v$ be any local completion of $\ell$ for which the natural inclusion $\rho^\prime:\Gam \to \bfG(\ell)\to \bfG(k)$ is not precompact. We also consider the corresponding representation $\rho:\Gam \to \bfG(k)\to \bfH(k)$ whose image is also not precompact. This section is then devoted to proving the following.

\begin{prop}
\label{prop:liftingmeasures}
Under the hypotheses of Theorem \ref{thm:main:homogeneous} or Theorem \ref{thm:main}, there is a $k$-rational faithful irreducible representation $\bfH \to \SL(V)$ on a $k$-vector space $V$ and a $W$-invariant measure on $(G\times\bbP(V))/\Gamma$ that projects to Haar measure on $G/\Gamma$.
\end{prop}

\begin{proof}
We retain all notation from prior subsections.
We note that, as $\mathbf{H}$ is semisimple, each of its $k$-rational representation is into $\SL(V) < \GL(V)$. We will construct a faithful irreducible representation supporting a $W$ invariant measure as required. By Proposition \ref{prop:homogen} we can assume that we have $W$-invariant, $W$-homogeneous measures $\mu_i$ on $G/\Gamma$ converging to Haar measure. As in the proof of Proposition \ref{prop:homogen} we know that each $\mu_i$ is Haar measure on some $x_i(S_i/\Gamma_i)$, where $S_i \in \calS$, $x_i \in G$, and $\Gamma_i= S_i \cap \Gamma$ is Zariski dense in the closed subgroup $S_i$. Thus $S_i$ is an $\ell$-defined subgroup of $\bfG$. Since our representation into $\bfG(k)$ is defined by localization at a place of $\ell$, it follows that $\rho(\Gamma_i)$ is contained in a proper $k$-algebraic subgroup $\bfL_i(k)$ of $\bfH(k)$. We pass to a subsequence where $m=\dim(S_i)=\dim(\bfL_i)$ is constant.

We first assume that $\bfH$ $k$-is simple, which is automatic when $G\neq \SO_0(3,1)$, and note that in this case the faithfulness of a $k$-linear representation is equivalent to its nontriviality. We will consider the semisimple case at the end of the proof, where faithfulness will require an argument.

Consider the $m^{\textrm{th}}$ exterior power $\wedge^m(\Ad): \bfH(k) \rightarrow \GL(\bigwedge^m \mathfrak{h})$ of the adjoint representation of $\bfH(k)$ on its Lie algebra $\mathfrak{h}$. The Lie algebra $\mathfrak{l}_i$ of $\bfL_i(k)$ determines a line $l_i$ in $\bigwedge^m \mathfrak{h}$. Since the stabilizer of $l_i$ in $\bfH(k)$ is the normalizer of $\bfL_i(k)$ and hence a proper subgroup, this line is never $\bfH(k)$-invariant. Therefore each $l_i$ projects nontrivially to some nontrivial irreducible summand of $\wedge^m(\Ad): \bH(k) \rightarrow \GL(\wedge^m \mathfrak{h})$. Since only finitely many irreducible representations can occur, one such irreducible representation $V$ occurs infinitely often. Passing to a further subsequence, we obtain an irreducible subrepresentation $V$ onto which each $l_i$ projects nontrivially. The point stabilizer of $l_i$ contains the image of $\bfL_i(k)$ and hence it contains $\rho(\Gamma_i)$.

Given the closed $W$-invariant subset $x_i(S_i/\Gamma_i)$, note that $l_i$ determines an invariant line bundle over $x_i(S_i/\Gamma_i)$ and therefore defines a measurable section
\[
s^i: x_i(S_i/\Gamma_i) \rightarrow (G\times \mathbb{P}(V))/\Gamma_i.
\]
Let $\mu_i$ be Haar measure on $x_i(S_i/\Gamma_i)$ and define $\nu_i = s^i_*{\mu}_i$. We then construct a $W$-ergodic, $W$-invariant measure on $(G \times \mathbb{P}(V))/\Gamma$ by taking $\nu$ to be any ergodic component of any weak-$*$ limit of the $\nu_i$ on $(G \times \mathbb{P}(V))/\Gamma$. Since the $\mu_i$ converge to Haar measure on $G/\Gamma$ and projection commutes with taking weak-$*$ limits, this implies that $\nu$ projects to Haar measure on $G / \Gam$ and completes the proof when $\bfH$ is simple.

For $G=\SO_0(3,1)$, the group $\bfH(k)$ need not be $k$-simple due to exceptional isomorphisms such as $\PO(4,k)\cong \PGL(2,k) \times \PGL(2,k)$ when $k$ is $p$-adic or complex or $\PO(2,2) \cong \PO(2,1) \times \PO(2,1)$. We therefore must find an irreducible representation $V$ on which $\bfH(k)$ acts faithfully and for which the above construction yields the necessary invariant measure. To this end, we need to consider cases when
\[
\bH(k)=\PGL_2(k)\times \PGL_2(k),
\]
where $k$ is $\bbR$, $\bbC$, or a nonarchimedean local field of characteristic zero.

Notice that $\Gamma$ is Zariski dense in the almost simple group $\bfG(\mathbb{R})$ and the groups $\Gamma_i$ have a proper Zariski closure. In particular, $\Gamma_i<\Gamma$ is not normal, and it follows from injectivity of $\rho$ that $\rho(\Gamma_i)$ is not contained in a direct factor of $\bfH(k)$.
As the Zariski closure of $\Gamma_i$ is almost simple in $\bfG(\ell)$ and $\rho$ is given by a field embedding, we conclude that $\rho(\Gamma_i)$ is contained in a conjugate of the diagonal subgroup $\Delta(\PGL_2(k))$ of $\PGL_2(k) \times \PGL_2(k)$ for all $i$.

We take the adjoint representation of $\PGL_2(k) \times \PGL_2(k)$ on $k^6$ and the diagonal three dimensional subspace $\Delta(k^3) < k^3 \oplus k^3$ stabilized by $\Delta(\PSL_2(k))$. A computation shows that $\bigwedge^3(k^3 \oplus k^3)$ splits as a direct sum of four irreducible representations of $\PGL_2(k) \times \PGL_2(k)$, two that are trivial and two that are isomorphic to the faithful representation $V(3,3)$ on $k^3 \otimes k^3$. One also checks that $\bigwedge^3(\Delta(k^3))$ projects nontrivially to each $V(3,3)$ (in fact, to all four summands). Taking $V=V(3,3)$ and arguing as above, we also produce a $W$-invariant measure on $(G\times \mathbb{P}(V))/\Gamma$ when $G=\SO_0(3,1)$. This completes the proof.
\end{proof}

\subsection{Compatibility}
\label{ssec:Compatible}

Let $G$, $U$, and $P = MAU$ be as defined in \S\ref{sec:FN}. Let $k$ be a local field and $\bH$ a $k$-algebraic group. Recall that the pair consisting of $k$ and $\bH$ is \emph{compatible} with $G$ if for every nontrivial $k$-subgroup $\bfJ<\bH$ and any continuous homomorphism $\tau:P \rightarrow N_{\bfH}(\bfJ)/\bfJ(k)$, where $N_{\bfH}(\bfJ)$ is the normalizer of $\bfJ$ in $\bfH$, we have that the Zariski closure of $\tau(U')$ coincides with the Zariski closure of $\tau(U)$ for every nontrivial subgroup $U'<U$.

Note that if the pair $(k',\bfH)$ is compatible, where $k'/k$ is a finite field extension, then the pair $(k,\bfH)$ is also compatible. Indeed, letting $\bfJ<\bH$ be a $k$-subgroup and $\tau:P \rightarrow N_{\bfH}(\bfJ)/\bfJ(k)$ be a continuous homomorphism, composing $\tau$ with the homomorphism $N_{\bfH}(\bfJ)/\bfJ(k) \rightarrow N_{\bfH}(\bfJ)/\bfJ(k')$ defines a continuous homomorphism $\tau':P \rightarrow N_{\bfH}(\bfJ)/\bfJ(k')$. Compatibility of $(k',\bfH)$ implies that the Zariski closure of $\tau(U')$ coincides with the Zariski closure of $\tau(U)$ for every nontrivial subgroup $U'<U$.

We note also that compatibility of $(k,\bfH)$ follows immediately if $U<\ker \tau$ for every $\tau$ as above. This is automatically the case when $k$ is nonarchimedean, since then the group $N_{\bfH}(\bfJ)/\bfJ(k)$ is totally disconnected while $U$ is connected.

\begin{lemma}
\label{lemma:compatible}
We retain the setting of \S\ref{sec:FN} and let $\bfH$ be the adjoint group of $\bfG(k)$ for $k$ a local field and $\ell \to k$ a field embedding. Then the pair consisting of $k$ and $\bfH$ is compatible with $G$.
\end{lemma}


\begin{proof}
As mentioned above, we can and hence will assume that $k$ is archimedean and, passing to a finite extension, in fact assume that $k=\mathbb{C}$. In particular we get that $\bfH$ is split and we identify it with $\PO(n+1, \bbC)$. In what follows, we identify algebraic subgroups of $\bfH$ with their complex points.

Assume we have a nontrivial algebraic subgroup $\bfJ<\PO(n+1, \bbC)$ and a continuous homomorphism $\tau:P\to \bfN/\bfJ$, where we denote $\bfN=\bfN_\bfH(\bfJ)$. Fix a nontrivial subgroup $U'<U$. Letting $\bfU,\bfU',\bfM,\bfA$ and $\bfP$ be the Zariski closures of the images of $U,U',M,A$ and $P$ in $\bfN/\bfJ$ correspondingly, we must show that $\bfU'=\bfU$. We can obviously assume that $\tau|_U$ is nontrivial.

We first make an observation for later use. Since $U$ is the derived subgroup of the solvable group $AU$, $\bfU$ is contained in the derived subgroup of the solvable group $\bfA\bfU$. It follows that $\bfU$ is a nontrivial unipotent group.

It is convenient to identify $U$ with the additive group $\bbR^{n-1}$ and $MA$ with its conformal group. In particular, we identify $M$ with $\SO(n-1,\bbR)$ and $A$ with the group of homotheties $\bbR^*$. Using transitivity of the action of $\SO(n-1,\bbR)$ on $\bbP^{n-2}(\bbR)$, one easily checks that every non-central subgroup of $P$ contains $U$. We conclude that $\tau$ has finite kernel.

We claim that $n=3$. We will assume that $n\geq 4$ and argue to show a contradiction. 

As $\tau(M)$ is locally isomorphic to the compact group $\SO(n-1,\bbR)$, we see that $\bfM$ is locally isomorphic to $\SO(n-1,\bbC)$. Thus $\bfM$ is almost simple and it normalizes the solvable group $\bfA\bfU$, as $M$ normalizes the solvable group $AU$ in $P$. Therefore $\bfM$ intersects the group $\bfA\bfU$ almost trivially. As $AU$ is not nilpotent, we get that $\bfA\bfU$ is not unipotent, thus $\rank \bfA\bfU \geq 1$. We conclude that $\rank \bfP\geq \rank \bfM+1$.
We note also that $\rank \bfP/\bfU=\rank \bfP$, as $\bfU$ is unipotent.

We let $\widetilde\bfU$ and $\widetilde\bfP$ be the corresponding preimages of $\bfU$ and $\bfP$ in $\bfN$ under the map $\bfN \to \bfN/\bfJ$. From the sequence of inequalities
\begin{align*}
\rank \bfH  &\geq \rank \widetilde\bfP \geq \rank\widetilde\bfP-\rank \widetilde\bfU\\
&= \rank \widetilde\bfP/ \widetilde\bfU= \rank \bfP/\bfU= \rank \bfP \\
&\geq \rank \bfM+1 =\rank \SO(n-1,\bbC)+1 \\
&=\rank \SO(n+1,\bbC)=\rank \bf H
\end{align*}
we deduce that $\rank \widetilde\bfU=0$.

Next we consider the identity component $\bfJ^0$ of $\bfJ$ and the identity component $\widetilde\bfU^0$ of $\widetilde\bfU$. We note that $\bfJ^0 \le \widetilde\bfU^0$ and that this is a proper inclusion by nontriviality of the connected group $\bfU$. As $\rank \widetilde\bfU=0$, we deduce that $\widetilde\bfU^0$ is a unipotent subgroup of $\bfN$. We conclude that both $\bfJ^0$ and $\widetilde\bfU^0$ are normal unipotent subgroups of $\widetilde\bfP$ with $\bfJ^0\lneq \widetilde\bfU^0$.

By \cite[\S3]{BorelTits}, we see that $\widetilde\bfP$ is contained in a maximal parabolic subgroup $\bfQ<\bfH$, as it contains a nontrivial normal unipotent subgroup. Note that the maximal parabolic subgroups of $\SO(n+1,\bbC)$ are the stabilizers of isotropic subspaces in $\bbC^{n+1}$. For a $k$-dimensional isotropic subspace, the semisimple part $\bfS<\bfQ$ of the Levi subgroup is locally isomorphic to $\SL_k(\bbC)\times \SO(n+1-2k,\bbC)$ noting that $2k\leq n+1$. Alternatively, this can be seen by removing a node from the Dynkin diagram associated with $\SO(n+1,\bbC)$. Almost-simplicity of $\bfM<\bfP=\widetilde\bfP/\bfJ$ implies that $\widetilde\bfP$ contains a group $\bfM'$ locally isomorphic to $\bfM$. Thus $\bfM'<\bfQ$, and upon conjugating $\bfS$ we can assume that $\bfM'<\bfS$. We conclude that $k=1$ and $\bfS$ is locally isomorphic to $\SO(n-1,\bbC)$. In particular, $\bfS=\bfM'<\widetilde\bfP$.

We denote the unipotent radical of $\bfQ$ by $\bfR$ and note that it has no proper, nontrivial $\bfS$-normalized subgroups. Indeed, this follows from the transitivity of the action of $\PO(n+1,\bbC)$ on $\bbP^{n-2}(\bbC)$. As $\bfS<\widetilde\bfP$ and $\bfR$ consists of all unipotent elements of $\bfQ$, we get that $\bfJ^0\lneq \widetilde\bfU^0$ are unipotent subgroups of $\bfR$ that are normalized by $\bfS$. We conclude that $\bfJ^0$ is trivial and $\widetilde\bfU^0=\bfR$.

As $\widetilde\bfP \le \bfQ$ contains both $\bfR$ and $\bfS$, we see that $\widetilde\bfP=\bfQ$. As $\bfJ^0$ is trivial, we obtain that $\bfJ$ is a finite normal subgroup of $\bfQ$. However, $\bfQ$ is a parabolic subgroup of the adjoint group $\bfH$, and hence it contains no nontrivial finite normal subgroup. This implies that $\bfJ$ is trivial, which gives the desired contradiction to the assumption $n\ge 4$.

We thus have $n=3$. That is, we have
\[
\bfH=\PO(4, \bbC) \simeq \PGL(2,\bbC)\times\PGL(2,\bbC).
\]
We will assume $\bfU' \lneq \bfU$ and derive a contradiction. By almost injectivity of $\tau$, $\bfU'$ is nontrivial unipotent subgroup and it follows that $\bfU \le \bfN/\bfJ$ is at least two dimensional. We thus can find a two dimensional unipotent subgroup $\bfV<\bfN$. Note that $\bfH$ has no three dimensional unipotent subgroup. It follows that $\bfJ$ has no nontrivial unipotent subgroup, thus $\bfV$ is the unipotent radical of $\bfV\bfJ$. As both $\bfV$ and $\bfJ$ are normal in $\bfV\bfJ$ and they have trivial intersection, it follows that they commute. We note that all two dimensional unipotent subgroups of $\bfH$ are conjugate and these are all unipotent radicals of Borel subgroups. It follows that $\bfN_{\bfH}(\bfV)$ is a Borel subgroup $\bfB<\bfH$. As $\bfJ$ normalizes $\bfV$, $\bfJ<\bfB$. Up to conjugation, we may assume that $\bfB$ is the standard Borel subgroup of $\PGL(2,\bbC)\times\PGL(2,\bbC)$ and it is easy to check that its unipotent radical is its own centralizer. It follows that $\bfJ<\bfV$. This forces $\bfJ$ to be trivial, as it has no unipotent subgroup. This gives the desired contradiction and thus finishes the proof.
\end{proof}

\begin{rem}
Note that $\PO(n+1, \bbC)$ can also be viewed as an algebraic group over $\bbR$ and for $k=\bbR$ it is \emph{not} compatible with $G = \SO_0(n,1)$. In the proof of arithmeticity, a Galois conjugate isomorphic to $\PO(n+1, \bbC)$ is naturally given the real Zariski topology, but for the purposes of proving Theorem \ref{theorem:superrigiditydichotomy} we may instead consider it in the complex Zariski topology.
\end{rem}

\section{Proof of Theorem \ref{theorem:superrigiditydichotomy}}
\label{section:proofs}

Throughout this section, we assume that we have a $W$-invariant measure $\nu$ on the bundle $(G\times \mathbb{P}(V))/\Gamma$ that projects to Haar measure on $G / \Gam$.

\subsection{From measures to measurable maps to varieties}
\label{subsec:H/L}

\noindent
This subsection converts our $W$-invariant measure into a measurable $\Gamma$-equivariant map between varieties.

\begin{prop}
\label{lemma:properness}
There exists a proper $k$-algebraic subgroup $\bfL < \bH$ and a measurable $W$-invariant, $\Gamma$-equivariant map $\phi: G \rightarrow \bH/ \bfL(k)$. We can also view $\phi$ as a measurable $\Gamma$-map from $W \backslash G$ to $\bH/\bfL(k)$.
\end{prop}

\begin{proof}
Via disintegration, the $W$-invariant measure $\nu$ on $(G\times \mathbb{P}(V))/\Gamma$ yields a $W$-invariant $\Gamma$-map
\[
\tilde{\phi}: G \rightarrow \mathcal{P}(\mathbb{P}(V)),
\]
where $\mathcal{P}(\bbP(V))$ is the space of probability measures on $\bbP(V)$. By \cite[Cor.\ 3.2.12 and Thm.\ 3.2.4]{Zimmer}, the image of this map lies in a single $\bH(k)$-orbit that can be identified with $\bH(k)/\tilde{L}$ for $\wt{L}$ a compact extension of the $k$-points of a $k$-algebraic subgroup of $\bH(k)$.
We thus get a $\Gamma$-equivariant map $W\backslash G\to \bH(k)/\tilde{L}$.

We claim that $\tilde{L}$ is not compact. If it were, we could find an $\bH(k)$-invariant metric on $\bH(k)/\tilde{L}$, but by \cite[Cor.\ 6.7]{Bader-Gelander} the action of $\Gamma$ on $W\backslash G$ is metrically ergodic (see \cite[Def.\ 6.5]{Bader-Gelander} for the definition) and thus the map $W\backslash G\to \bH(k)/\tilde{L}$ would be essentially constant with $\Gamma$-invariant image. This would contradict the assumption that $\rho:\Gamma\to \bH(k)$ is unbounded, hence $\tilde{L}$ cannot be compact.

Let $\bfL$ be the Zariski closure of $\tilde L$. Then \cite[Prop.\ 3.2.15]{Zimmer} implies that $\bfL$ is a proper $k$-subgroup of $\bH$. We are then done by composing $G\to \bH(k)/\tilde{L}$ with the natural map $\bH(k)/\tilde{L} \to \bH/\bfL(k)$.
\end{proof}

\begin{rem}
One can also prove the group $\tilde{L}$ is noncompact by showing nontriviality of the Lyapunov spectrum of the $W$-action on the bundle $(G\times \mathbb{P}(V))/\Gamma$ using \cite[Thm.\ V.5.15]{MargulisBook}. This is delicate when $\Gamma<G$ is nonuniform, relying on its weak cocompactness and integrability of the standard cocycle $\alpha: G \times G/\Gamma \rightarrow \Gamma$.
\end{rem}

Note that the subgroup $\bfL$ of $\bfH$ might be a normal (even trivial) subgroup, or when $\bfH$ is semisimple but not simple, it might consist of a nontrivial factor group. In the latter case the $\bH$ action on $\bH/\bfL$ is not effective. However these caveats do not effect our proof.

\subsection{Algebraic representations}
\label{ssec:AlgRep}

In this subsection, we introduce the ideas from the work of Bader and Furman \cite{BaderFurman} used in the proof of our superrigidity theorem.

Let $k$ be a local field, fix a $k$-algebraic group $\bH$, and let $H = \bH(k)$ denote the $k$-points of $\bH$. To start, let $G$ be a locally compact second countable group, $\Gam < G$ be a lattice, and $\rho : \Gam \to H$ be a Zariski dense representation.

Given a closed subgroup $T < G$, a \emph{$T$-algebraic representation of $G$} consists of:
\begin{itemize}

\item a $k$-algebraic group $\bfI$,

\item a $k$-$(\bH \times \bfI)$-algebraic variety $\bfV$, which is considered as a left $\bH$-space and a right $\bfI$-space on which the $\bfI$-action is faithful,

\item a Zariski dense homomorphism $\tau : T \to \bfI(k)$,

\item an \emph{algebraic representation} of $G$ on $\bfV$, i.e., an almost-everywhere defined measurable map $\phi : G \to \bfV(k)$
such that
\[
\phi(t g \gam^{-1}) = \rho(\gam) \phi(g) \tau(t)^{-1}
\]
for every $\gam \in \Gam$, every $t \in T$, and almost every $g \in G$.

\end{itemize}
We denote the data for a $T$-algebraic representation of $G$ by $\bfI_\bfV$, $\tau_\bfV$, and $\phi_\bfV$.

A $T$-algebraic representation is called \emph{coset $T$-algebraic} when $\bfV$ is the coset space $\bH / \bfJ$ for some $k$-algebraic subgroup $\bfJ$ of $\bH$, and $\bfI$ is a $k$-subgroup of $N_\bH(\bfJ) / \bfJ$, where $N_\bH(\bfJ)$ denotes the normalizer of $\bfJ$ in $\bH$. Given another $T$-algebraic representation $\bfU$, let $\bfI_{\bfU, \bfV}$ be the Zariski closure of $(\tau_\bfU \times \tau_\bfV)(T)$ in $\bfI_\bfU \times \bfI_\bfV$. Then a \emph{morphism} $\pi : \bfU \to \bfV$ is an $(\bH \times \bfI_{\bfU, \bfV})$-equivariant $k$-regular morphism such that $\phi_\bfV$ agrees almost everywhere with $\pi \circ \phi_\bfU$. Recall that an \emph{initial object} in a category is an object that has exactly one morphism to all other objects in the category. The proof of our superrigidity theorem uses the following.

\begin{thm}[Thm.\ 4.3 \cite{BaderFurman}]\label{thm:BF}
The collection of $T$-algebraic representations of $G$ forms a category. If the $T$-action on $G / \Gam$ is weakly mixing, then this category has an initial object and this initial object is a coset $T$-algebraic representation.
\end{thm}

An initial object is characterized by the fact that $\bfJ$ is the minimal subgroup, up to conjugacy, that can arise as a stabilizer in any coset $T$-algebraic representation in the category.

Though not stated explicitly, the following is also implicit in \cite{BaderFurman}. Given two subgroups $S$ and $T$ of $G$, we say that their initial objects $\phi_S : G \to \bfV(k)$ and $\phi_T : G \to \bfW(k)$ {\em have the same map} if $\bfV=\bfW$ as $k$-varieties and if $\phi_S$ and $\phi_T$ agree away from a set of measure zero.

\begin{lemma}\label{lem:initobj}
Assume that the action of $T$ on $G/\Gam$ is weakly mixing. Then initial objects for $T$ and $N_G(T)$ have the same map. Moreover, the initial object for $T$ and for the iterated normalizer $N_G(N_G(\dots (N_G(T))\dots))$ have the same map.
\end{lemma}

\begin{proof}
For the first claim, the forward direction is the content of \cite[Thm.\ 4.6]{BaderFurman}. For the backward direction, if $\phi : G \to \bfH/\bfJ(k)$ is an initial object in the category of $N_G(T)$-algebraic representations with associated homomorphism $\tau$ from $T$ to $N_\bfH(\bfJ)/\bfJ(k)$, then $\phi$ and $\tau\vert_T$ form a $T$-algebraic representation and this representation must be initial by minimality. Indeed, otherwise another application of the forward direction contradicts minimality of $\bfJ$. The second claim follows immediately from the first.
\end{proof}

\subsection{From measurable maps to extension of homomorphisms}
\label{subsec:extends}

We now complete the proof of Theorem \ref{theorem:superrigiditydichotomy}. More specifically, we show that the existence of the map $\phi: G \rightarrow \bH/ \bfL(k)$ from Proposition \ref{lemma:properness} implies that the representation $\rho$ of $\Gamma$ extends to $G$.

\begin{proof}[Proof of Theorem \ref{theorem:superrigiditydichotomy}]
Observe that the action of $G$ on $G/\Gamma$ is mixing by the Howe--Moore theorem. In particular, the action of each noncompact subgroup of $G$ is weakly mixing on $G/\Gamma$. This allows us to apply freely the discussion and results of \S\ref{ssec:AlgRep} regarding $T$-algebraic representations of $G$ for an arbitrary noncompact closed subgroup $T$ of $G$.

Recall our setting from \S\ref{sec:FN}, and first consider $T=U'=U\cap W$, which is noncompact. Given an initial object in the category of $U^\prime$-algebraic representations of $G$, Theorem \ref{thm:BF} implies that there is a $k$-algebraic subgroup $\bfJ$ of $\bH$ such that this object is a measurable map $\Psi:G \rightarrow \bH/\bfJ(k)$ that is $(U^\prime \times \Gamma)$-equivariant for a continuous homomorphism $\tau: U^\prime \to N_{\bfH}(\bfJ)/\bfJ(k)$. Since $U^\prime$ is normal in $U$ and $U$ is normal in $P$, Lemma \ref{lem:initobj} implies that $\tau$ extends to a continuous homomorphism $\tau:P \to N_{\bfH}(\bfJ)/\bfJ(k)$ making the map $\Psi$ an initial object in the category of $P$-algebraic representations of $G$.

We claim that $\bfJ$ is trivial. Assume this is not the case. Since the pair consisting of $k$ and $\bfH$ is compatible with $G$ we know that the Zariski closure of $\tau(U')$ coincides with the Zariski closure of $\tau(U)$. We note that the $W$-invariant map $\phi$ is also $U'$-invariant, as $U'<W$, thus it factors via
$\Psi:G\to \bH/\bfJ(k)$ and via
\[
G\to \bH/\bfJ(k) \to (\bH/\bfJ)/\overline{\tau(U')}(k)= (\bH/\bfJ)/\overline{\tau(U)}(k)
\]
by $U'$-invariance, where $\overline{\tau(U')}$ and $\overline{\tau(U)}$ are the Zariski closures. Then, $\Psi$ is $U$-equivariant, so the latter composed map is $U$-invariant, and it follows that $\phi$ is also $U$-invariant. Since $\phi$ is also $W$-invariant and $\langle U, W \rangle = G$, we obtain that $\phi : G \to \bH / \bfL(k)$ is an essentially constant $\Gam$-equivariant map, hence $\rho(\Gam)$ has a fixed point on $\bH / \bfL(k)$. This is impossible since $\rho(\Gamma)$ is Zariski dense in $\bH$ and $\bfL$ is a proper algebraic subgroup of the connected adjoint group $\bH$. We conclude that $\bfJ$ is indeed trivial.

Since $\bfJ$ is trivial and $A < P$, we view $\tau$ as a morphism $\tau:P \to \bfH$ and $\Psi$ as a $(P\times \Gamma)$-map $\Psi: G \rightarrow \bH(k)$. In particular, $\Psi$ is $A$-equivariant via the homomorphism $\tau|_A$, and therefore must be an initial object for the category of $A$-algebraic representations by Lemma \ref{lem:initobj}. Once again, Lemma \ref{lem:initobj} implies that $\tau|_A$ extends to a homomorphism $\tau':N_G(A) \rightarrow \bH(k)$ for which $\Psi$ is $N_G(A)$-equivariant, where $N_G(A)$ is the normalizer of $A$ in $G$.

Notice that $N_G(A)$ contains a Weyl element $w$ for $A$ and hence $\langle P, N_G(A) \rangle = G$. Since $\Psi$ is equivariant for both $P$ and $N_G(A)$, using \cite[Prop.\ 5.1]{BaderFurman} and following the end of the proof of \cite[Thm.\ 1.3]{BaderFurman}, we deduce that $\rho : \Gam \to \bfH(k)$ extends to a continuous homomorphism $\wh{\rho} : G \to \bfH(k)$. This proves the theorem.
\end{proof}

\section{Theorem \ref{thm:dehnfilling} and final remarks}\label{sec:Final}

In this section, we adapt the proof of Theorem \ref{thm:main} to prove Theorem \ref{thm:dehnfilling}, then make some final remarks and ask some questions related to our main results.

\subsection{The proof of Theorem \ref{thm:dehnfilling}}\label{ssec:FinalA}

Let $M$ and $N$ be connected, orientable hyperbolic $3$-manifolds of finite volume, and suppose that $N$ is obtained by Dehn filling on a nonempty subset of the torus cusps of $M$. If $\Gam = \pi_1(M)$ and $\Lam = \pi_1(N)$, the map $M \to N$ induced by the filling determines a surjective homomorphism $\rho : \Gam \to \Lam$. Since $\Gam$ and $\Lam$ are naturally lattices in $\SO_0(3,1)$, we can consider $\rho$ as a homomorphism from $\Gam$ to $\SO_0(3,1)$ with $\rho(\Gam)$ isomorphic to $\Lam$. Note that $\rho$ has nontrivial kernel.

If either of $M$ or $N$ is nonarithmetic, then Theorem \ref{thm:main} immediately implies that $M$ and $N$ contain only finitely many totally geodesic surfaces. However, Theorem \ref{thm:main} is not applicable with both $M$ and $N$ are arithmetic. Before giving the proof of Theorem \ref{thm:dehnfilling}, we give an example to show that the theorem is indeed nontrivial.

\begin{ex}
Let $N$ be the complement in $S^3$ of the $3$-chain link, which is also called $6_1^3$ in the Rolfsen tables \cite{Rolfsen}. Then $N$ is arithmetic \cite[\S 9.2]{MaclachlanReid}. Moreover, $N$ is obtained from trivial Dehn filling on one component of the four component arithmetic link complement given in \cite[Ex.\ 6.8.10]{Thurston}, also known as L12n2210. See Figure \ref{fig:Links}.
\begin{figure}
\centering
\definecolor{linkcolor0}{rgb}{0.85, 0.15, 0.15}
\definecolor{linkcolor1}{rgb}{0.15, 0.15, 0.85}
\definecolor{linkcolor2}{rgb}{0.15, 0.85, 0.15}
\begin{tikzpicture}[line width=1.0, line cap=round, line join=round,scale=0.5]
  \begin{scope}[color=linkcolor0]
    \draw (2.00, 1.61) .. controls (1.74, 1.61) and (1.58, 1.33) .. (1.58, 1.03);
    \draw (1.58, 1.03) .. controls (1.58, 0.16) and (3.45, 0.16) ..
          (4.95, 0.16) .. controls (6.39, 0.16) and (8.31, 0.16) .. (8.32, 0.75);
    \draw (8.33, 1.03) .. controls (8.34, 1.36) and (7.99, 1.56) .. (7.63, 1.56);
    \draw (7.63, 1.56) .. controls (5.87, 1.57) and (4.11, 1.59) .. (2.35, 1.60);
  \end{scope}
  \begin{scope}[color=linkcolor1]
    \draw (4.88, 5.82) .. controls (4.37, 6.25) and (3.18, 4.90) ..
          (2.27, 3.87) .. controls (1.35, 2.85) and (0.16, 1.50) ..
          (0.64, 1.02) .. controls (0.89, 0.77) and (1.27, 0.70) .. (1.49, 0.93);
    \draw (1.69, 1.15) .. controls (1.83, 1.30) and (1.97, 1.45) .. (2.10, 1.60);
    \draw (2.10, 1.60) .. controls (2.90, 2.47) and (3.69, 3.34) .. (4.49, 4.21);
    \draw (4.81, 4.56) .. controls (5.17, 4.95) and (5.25, 5.52) .. (4.88, 5.82);
  \end{scope}
  \begin{scope}[color=linkcolor2]
    \draw (4.93, 5.89) .. controls (5.15, 6.14) and (6.49, 4.83) ..
          (7.47, 3.87) .. controls (8.48, 2.89) and (9.79, 1.61) ..
          (9.29, 1.03) .. controls (9.03, 0.73) and (8.60, 0.64) .. (8.33, 0.90);
    \draw (8.33, 0.90) .. controls (8.14, 1.08) and (7.96, 1.25) .. (7.77, 1.43);
    \draw (7.46, 1.73) .. controls (6.52, 2.61) and (5.59, 3.50) .. (4.65, 4.39);
    \draw (4.65, 4.39) .. controls (4.31, 4.71) and (4.40, 5.26) .. (4.72, 5.64);
  \end{scope}
\end{tikzpicture}
\definecolor{linkcolor0}{rgb}{0.85, 0.15, 0.15}
\definecolor{linkcolor1}{rgb}{0.15, 0.15, 0.85}
\definecolor{linkcolor2}{rgb}{0.15, 0.85, 0.15}
\definecolor{linkcolor3}{rgb}{0.15, 0.85, 0.85}
\begin{tikzpicture}[line width=1.0, line cap=round, line join=round,scale=0.5]
  \begin{scope}[color=linkcolor0]
    \draw (2.00, 1.61) .. controls (1.74, 1.61) and (1.58, 1.33) .. (1.58, 1.03);
    \draw (1.58, 1.03) .. controls (1.58, 0.16) and (3.45, 0.16) ..
          (4.95, 0.16) .. controls (6.39, 0.16) and (8.31, 0.16) .. (8.32, 0.75);
    \draw (8.33, 1.03) .. controls (8.34, 1.36) and (7.98, 1.56) .. (7.62, 1.56);
    \draw (7.62, 1.56) .. controls (7.30, 1.56) and (6.98, 1.57) .. (6.66, 1.57);
    \draw (6.66, 1.57) .. controls (5.58, 1.58) and (4.50, 1.59) .. (3.42, 1.59);
    \draw (2.96, 1.60) .. controls (2.75, 1.60) and (2.53, 1.60) .. (2.32, 1.60);
  \end{scope}
  \begin{scope}[color=linkcolor1]
    \draw (4.87, 5.83) .. controls (4.37, 6.25) and (3.18, 4.90) ..
          (2.27, 3.87) .. controls (1.35, 2.85) and (0.16, 1.50) ..
          (0.64, 1.02) .. controls (0.89, 0.77) and (1.27, 0.70) .. (1.49, 0.93);
    \draw (1.69, 1.15) .. controls (1.83, 1.30) and (1.97, 1.45) .. (2.10, 1.60);
    \draw (2.10, 1.60) .. controls (2.45, 1.99) and (2.80, 2.37) .. (3.15, 2.75);
    \draw (3.15, 2.75) .. controls (3.34, 2.95) and (3.53, 3.16) .. (3.71, 3.36);
    \draw (4.01, 3.68) .. controls (4.16, 3.85) and (4.32, 4.02) .. (4.47, 4.19);
    \draw (4.79, 4.54) .. controls (5.15, 4.93) and (5.25, 5.52) .. (4.87, 5.83);
  \end{scope}
  \begin{scope}[color=linkcolor2]
    \draw (4.93, 5.89) .. controls (5.15, 6.14) and (6.50, 4.82) ..
          (7.47, 3.87) .. controls (8.48, 2.89) and (9.79, 1.61) ..
          (9.29, 1.03) .. controls (9.03, 0.73) and (8.60, 0.64) .. (8.33, 0.90);
    \draw (8.33, 0.90) .. controls (8.14, 1.07) and (7.95, 1.25) .. (7.76, 1.43);
    \draw (7.44, 1.73) .. controls (7.23, 1.92) and (7.03, 2.11) .. (6.82, 2.31);
    \draw (6.47, 2.64) .. controls (6.16, 2.93) and (5.85, 3.22) .. (5.54, 3.51);
    \draw (5.54, 3.51) .. controls (5.23, 3.80) and (4.93, 4.08) .. (4.63, 4.36);
    \draw (4.63, 4.36) .. controls (4.28, 4.69) and (4.37, 5.26) .. (4.71, 5.65);
  \end{scope}
  \begin{scope}[color=linkcolor3]
    \draw (3.15, 2.90) .. controls (3.14, 3.26) and (3.48, 3.51) .. (3.85, 3.51);
    \draw (3.85, 3.51) .. controls (4.33, 3.51) and (4.81, 3.51) .. (5.29, 3.51);
    \draw (5.75, 3.51) .. controls (6.27, 3.51) and (6.63, 3.02) .. (6.64, 2.47);
    \draw (6.64, 2.47) .. controls (6.65, 2.23) and (6.65, 1.99) .. (6.66, 1.75);
    \draw (6.66, 1.40) .. controls (6.68, 0.72) and (5.74, 0.72) ..
          (4.93, 0.71) .. controls (4.10, 0.70) and (3.19, 0.87) .. (3.17, 1.60);
    \draw (3.17, 1.60) .. controls (3.17, 1.90) and (3.16, 2.21) .. (3.16, 2.52);
  \end{scope}
\end{tikzpicture}
\caption{The arithmetic links $6_1^3$ and L12n2210}\label{fig:Links}
\end{figure}
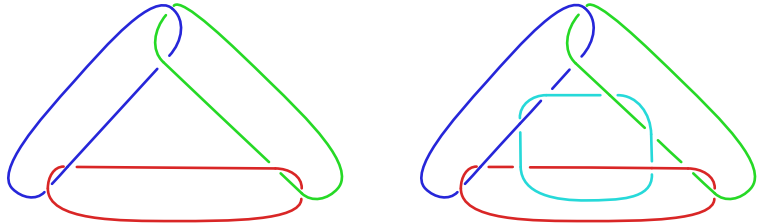
Using symmetries of the link diagrams, one sees that there are totally geodesic $4$-punctured spheres in $M$ that fill to become totally geodesic $3$-punctured spheres in $N$. Therefore, the collection of totally geodesic surfaces in $M$ that fill to a totally geodesic surface in $N$ is nonempty.
\end{ex}

\noindent
One can easily find other examples of this nature. We now prove Theorem \ref{thm:dehnfilling}.

\begin{proof}[Proof of Theorem \ref{thm:dehnfilling}]
Given $M$ and $N$ as in the statement of the theorem, let $\rho : \Gam \to \SO_0(3,1)$ be the representation defined above. We will prove that if infinitely many totally geodesic surfaces in $N$ are the images of totally geodesic surfaces in $M$, then $\rho$ must extend to a homomorphism $\SO_0(3,1) \to \PO_0(3,1)$. Since $\rho$ has a nontrivial kernel, this is impossible. We let $G$ denote $\SO_0(3,1)$ containing $\Gam$ as a lattice and $H$ denote $\PO_0(3,1)$ as the target for $\rho$. Note that $\rho$ has unbounded and Zariski dense image.

To apply Theorem \ref{theorem:superrigiditydichotomy}, we take $W = \SO_0(2,1)$ and then must produce an $H$-representation $V$ and a $W$-invariant measure $\nu$ on the bundle $(G \times \bbP(V)) / \Gam$ that projects to Haar measure on $G / \Gam$. Then Theorem \ref{theorem:superrigiditydichotomy} implies that $\rho$ extends to a representation of $G$, which gives the desired contradiction.

As in the proof of Theorem \ref{thm:main}, we produce this measure by finding an invariant line bundle over each closed $W$-orbit in $G/\Gam$. Let $V$ be a nontrivial, faithful, irreducible summand of the third exterior power of $\fraks\frako(3,1)$ with the adjoint action of $H$. Let $\{\Del_i\}$ be Fuchsian subgroups of $\Gam$ associated with totally geodesic surfaces of $M$ that remain totally geodesic under Dehn filling. Then $\rho(\Del_i)$ is a Fuchsian subgroup of $\Lam = \rho(\Gam)$, and hence it is contained in a subgroup $W_i$ of $H$ conjugate to the standard embedding of $\Isom^+(\bbH^2)$ in $\SO_0(3,1)$. Moreover, $W_i$ stabilizes a line in $V$ under the adjoint action, and the construction of $\nu$ proceeds exactly as in the proof of Proposition \ref{prop:liftingmeasures}. Thus Theorem \ref{theorem:superrigiditydichotomy} applies and the proof is complete.
\end{proof}

\begin{rem}
We also note that it is frequently the case that a $\pi_1$-injective surface in a $3$-manifold remains $\pi_1$-injective under Dehn filling (e.g., see \cite{CooperLong}). Therefore, infinitely many totally geodesic surfaces in $M$ may descend to $\pi_1$-injective surfaces in $N$. Our results say that these surfaces are very rarely totally geodesic.
\end{rem}

\subsection{Final remarks and questions}\label{ssec:FinalB}

We begin by noting that every known construction of a nonarithmetic hyperbolic $n$-manifold for $n \ge 4$ contains a totally geodesic hypersurface. Theorem \ref{thm:main} implies that the set of such hypersurfaces is always finite.

\begin{qtn}
For each $n \ge 4$ and $1 \le k < n-1$, does there exist a nonarithmetic hyperbolic $n$-manifold for which the set of totally geodesic subspaces of codimension $k$ is empty?
\end{qtn}

\noindent Answering this question in the positive will require a genuinely new construction of hyperbolic manifolds.
Perhaps more tractable is:

\begin{qtn}
For each $m \ge 1$, is there a hyperbolic $3$-manifold containing exactly $m$ totally geodesic surfaces?
\end{qtn}

\noindent
Finally, we ask about asymptotic properties of our results:

\begin{qtn}
Let $H_{n,m}(v)$ be the number of lattices $\Gam < \SO_0(n,1)$ such that $\bbH^n / \Gam$ contains exactly $m$ totally geodesic hypersurfaces and $\mathrm{vol}(\bbH^n / \Gam) < v$. What is the growth type of $H_{n,m}(v)$ as a function of $v$?
\end{qtn}

\begin{rem}
In part (3) of Corollary \ref{cor:kleinian}, we note that having infinitely many totally geodesic hypersurfaces gives a geometric characterization of arithmeticity in even dimensions. For $n \neq 3,7$ odd, there is a similar statement. In this case, every arithmetic hyperbolic manifold either contains maximal totally geodesic submanifold of codimension $1$ or $2$ (see \cite{Meyer}), hence having infinitely many such submanifolds again characterizes arithmeticity. There are arithmetic and nonarithmetic hyperbolic $3$-manifolds that contain no totally geodesic surfaces, so such a characterization is not possible, see \cite[\S6.1]{FLMS} for discussion and examples. For $n = 7$, one must classify the geodesic submanifolds of the arithmetic manifolds arising from triality; for those arithmetic manifolds not arising from that construction, the situation is the same as for other odd dimensions greater than $3$.
\end{rem}
\bibliographystyle{abbrv}
\bibliography{Biblio}

\begin{thebibliography}{10}

\bibitem{BaderFurman}
U.~Bader and A.~Furman.
\newblock An extension of {M}argulis' super-rigidity theorem.
\newblock Preprint: arXiv:1810.01608.

\bibitem{BFS}
U.~Bader, A.~Furman, and R.~Sauer.
\newblock Integrable measure equivalence and rigidity of hyperbolic lattices.
\newblock {\em Invent. Math.}, 194(2):313--379, 2013.

\bibitem{Bader-Gelander}
U.~Bader and T.~Gelander.
\newblock Equicontinuous actions of semisimple groups.
\newblock {\em Groups Geom. Dyn.}, 11(3):1003--1039, 2017.

\bibitem{BO}
Y.~Benoist and H.~Oh.
\newblock Geodesic planes in geometrically finite acylindrical 3-manifolds.
\newblock Preprint: arXiv:1802.04423.

\bibitem{BorelCrelle}
A.~Borel.
\newblock Density and maximality of arithmetic subgroups.
\newblock {\em J. Reine Angew. Math.}, 224:78--89, 1966.

\bibitem{BorelTits}
A.~Borel and J.~Tits.
\newblock {\'E}l\'ements unipotents et sous-groupes paraboliques de groupes
  r\'eductifs. {I}.
\newblock {\em Invent. Math.}, 12:95--104, 1971.

\bibitem{BrownFisherHurtado}
A.~Brown, D.~Fisher, and S.~Hurtado.
\newblock Zimmer's conjecture: subexponential growth, measure rigidity, and
  strong property $(\mathrm{T})$.
\newblock Preprint: arXiv:1608.04995, 2016.

\bibitem{BrownFisherHurtado2}
A.~Brown, D.~Fisher, and S.~Hurtado.
\newblock Zimmer's conjecture for actions of $\mathrm{SL}(m,\mathbb{Z})$.
\newblock Preprint: arXiv:1710.02735, 2017.

\bibitem{Calabi}
E.~Calabi.
\newblock On compact, {R}iemannian manifolds with constant curvature. {I}.
\newblock In {\em Proc. {S}ympos. {P}ure {M}ath., {V}ol. {III}}, pages
  155--180. American Mathematical Society, Providence, R.I., 1961.

\bibitem{CooperLong}
D.~Cooper and D.~D. Long.
\newblock Some surface subgroups survive surgery.
\newblock {\em Geom. Topol.}, 5:347--367, 2001.

\bibitem{DaniMargulis}
S.~G. Dani and G.~A. Margulis.
\newblock Limit distributions of orbits of unipotent flows and values of
  quadratic forms.
\newblock In {\em I. {M}. {G}el'fand {S}eminar}, volume~16 of {\em Adv. Soviet
  Math.}, pages 91--137. Amer. Math. Soc., 1993.

\bibitem{Curt}
K.~Delp, D.~Hoffoss, and J.~Manning.
\newblock Problems in {G}roups, {G}eometry, and {T}hree-{M}anifolds.
\newblock Preprint: arXiv:1512.04620.

\bibitem{Einsiedler}
M.~Einsiedler.
\newblock Ratner's theorem on {${\rm SL}(2,\Bbb R)$}-invariant measures.
\newblock {\em Jahresber. Deutsch. Math.-Verein.}, 108(3):143--164, 2006.

\bibitem{FLMS}
D.~Fisher, J.-F. Lafont, N.~Miller, and M.~Stover.
\newblock Finiteness of maximal geodesic submanifolds in hyperbolic hybrids.
\newblock Preprint: arXiv:1802.04619.

\bibitem{Garland}
H.~Garland.
\newblock On deformations of lattices in {L}ie groups.
\newblock In {\em Algebraic {G}roups and {D}iscontinuous {S}ubgroups ({P}roc.
  {S}ympos. {P}ure {M}ath., {B}oulder, {C}olo., 1965)}, pages 400--404. Amer.
  Math. Soc., Providence, R.I., 1966.

\bibitem{Gelander-Levit}
T.~Gelander and A.~Levit.
\newblock Counting commensurability classes of hyperbolic manifolds.
\newblock {\em Geom. Funct. Anal.}, 24(5):1431--1447, 2014.

\bibitem{GPS}
M.~Gromov and I.~Piatetski-Shapiro.
\newblock Nonarithmetic groups in lobachevsky spaces.
\newblock {\em Inst. Hautes \'Etudes Sci. Publ. Math.}, (66):93--103, 1988.

\bibitem{LO}
M.~Lee and H.~Oh.
\newblock Orbit closures of unipotent flows for hyperbolic manifolds with
  fuchsian ends.
\newblock Preprint: arXiv:1902.06621.

\bibitem{MRTG}
C.~Maclachlan and A.~W. Reid.
\newblock Commensurability classes of arithmetic {K}leinian groups and their
  {F}uchsian subgroups.
\newblock {\em Math. Proc. Cambridge Philos. Soc.}, 102(2):251--257, 1987.

\bibitem{MaclachlanReid}
C.~Maclachlan and A.~W. Reid.
\newblock {\em The arithmetic of hyperbolic 3-manifolds}, volume 219 of {\em
  Graduate Texts in Mathematics}.
\newblock Springer-Verlag, 2003.

\bibitem{Marden}
A.~Marden.
\newblock {\em Outer circles}.
\newblock Cambridge University Press, 2007.
\newblock An introduction to hyperbolic 3-manifolds.

\bibitem{MM}
G.~Margulis and A.~Mohammadi.
\newblock Arithmeticity of hyperbolic 3-manifolds containing infinitely many
  totally geodesic surfaces.
\newblock Preprint: arXiv:1902.07267.

\bibitem{MargulisICM}
G.~A. Margulis.
\newblock Discrete groups of motions of manifolds of nonpositive curvature.
\newblock In {\em Proceedings of the {I}nternational {C}ongress of
  {M}athematicians ({V}ancouver, {B}.{C}., 1974), {V}ol. 2}, pages 21--34.
  Canad. Math. Congress, Montreal, Que., 1975.

\bibitem{MargulisSuperrigidity}
G.~A. Margulis.
\newblock Arithmeticity of the irreducible lattices in the semisimple groups of
  rank greater than {$1$}.
\newblock {\em Invent. Math.}, 76(1):93--120, 1984.

\bibitem{MargulisBook}
G.~A. Margulis.
\newblock {\em Discrete subgroups of semisimple Lie groups}, volume~17 of {\em
  Ergebnisse der Mathematik und ihrer Grenzgebiete (3)}.
\newblock Springer-Verlag, 1991.

\bibitem{MMO}
C.~T. McMullen, A.~Mohammadi, and H.~Oh.
\newblock Geodesic planes in hyperbolic 3-manifolds.
\newblock {\em Invent. Math.}, 209(2):425--461, 2017.

\bibitem{MMO2}
C.~T. McMullen, A.~Mohammadi, and H.~Oh.
\newblock Geodesic planes in the convex core of an acylindrical 3-manifold.
\newblock Preprint: arXiv:1802.03853, 2018.

\bibitem{MR}
D.~B. McReynolds and A.~W. Reid.
\newblock The genus spectrum of hyperbolic 3--manifolds.
\newblock {\em Math. Res. Lett.}, 21:169--185, 2014.

\bibitem{Meyer}
J.~S. Meyer.
\newblock Totally geodesic spectra of arithmetic hyperbolic spaces.
\newblock {\em Trans. Amer. Math. Soc.}, 369(11):7549--7588, 2017.

\bibitem{MS}
S.~Mozes and N.~Shah.
\newblock On the space of ergodic invariant measures of unipotent flows.
\newblock {\em Ergodic Theory Dynam. Systems}, 15(1):149--159, 1995.

\bibitem{RaghunathanRigid}
M.~S. Raghunathan.
\newblock Cohomology of arithmetic subgroups of algebraic groups. {II}.
\newblock {\em Ann. of Math. (2)}, 87:279--304, 1967.

\bibitem{RatnerDuke}
M.~Ratner.
\newblock Raghunathan's topological conjecture and distributions of unipotent
  flows.
\newblock {\em Duke Math. J.}, 63(1):235--280, 1991.

\bibitem{RatnerOrbitClosure}
M.~Ratner.
\newblock Raghunathan's topological conjecture and distributions of unipotent
  flows.
\newblock {\em Duke Math. J.}, 63(1):235--280, 1991.

\bibitem{ReidKnot}
A.~W. Reid.
\newblock Arithmeticity of knot complements.
\newblock {\em J. London Math. Soc. (2)}, 43(1):171--184, 1991.

\bibitem{ReidTG}
A.~W. Reid.
\newblock Totally geodesic surfaces in hyperbolic {$3$}-manifolds.
\newblock {\em Proc. Edinburgh Math. Soc. (2)}, 34(1):77--88, 1991.

\bibitem{Rolfsen}
D.~Rolfsen.
\newblock {\em Knots and links}, volume~7 of {\em Mathematics Lecture Series}.
\newblock Publish or Perish, 1990.

\bibitem{Selberg}
A.~Selberg.
\newblock On discontinuous groups in higher-dimensional symmetric spaces.
\newblock In {\em Contributions to function theory (internat. {C}olloq.
  {F}unction {T}heory, {B}ombay, 1960)}, pages 147--164. Tata Institute of
  Fundamental Research, Bombay, 1960.

\bibitem{Thurston}
W.~Thurston.
\newblock Geometry and topology of three-manifolds.
\newblock http://http://library.msri.org/books/gt3m/.

\bibitem{VinbergDef}
E.~B. Vinberg.
\newblock Rings of definition of dense subgroups of semisimple linear groups.
\newblock {\em Izv. Akad. Nauk SSSR Ser. Mat.}, 35:45--55, 1971.

\bibitem{ZimmerCSR}
R.~J. Zimmer.
\newblock Strong rigidity for ergodic actions of semisimple {L}ie groups.
\newblock {\em Ann. of Math. (2)}, 112(3):511--529, 1980.

\bibitem{Zimmer}
R.~J. Zimmer.
\newblock {\em Ergodic theory and semisimple groups}, volume~81 of {\em
  Monographs in Mathematics}.
\newblock Birkh\"auser Verlag, Basel, 1984.

\end{thebibliography}

\end{document}